\documentclass[english]{amsart}

\usepackage{esint}
\usepackage[svgnames]{xcolor} 
\usepackage[colorlinks,citecolor=red,pagebackref,hypertexnames=false,breaklinks]{hyperref}
\usepackage{pgf,tikz}
\usepackage{pdfsync}

\usepackage{dsfont}
\usepackage{url}
\usepackage[utf8]{inputenc}
\usepackage[T1]{fontenc}
\usepackage{lmodern}
\usepackage{babel}
\usepackage{mathtools}  
\usepackage{amssymb}
\usepackage{lipsum}
\usepackage{mathrsfs}
\usepackage{color}

\newtheorem{theorem}{Theorem}[section]

\newtheorem{lemma}{Lemma}[section]

\newtheorem{corollary}{Corollary}[section]

\newtheorem{remark}{Remark}[section]

\numberwithin{equation}{section}

\title[two sided Gaussian estimates]{Two-sided Gaussian bounds  for  fundamental solutions of non-divergence form parabolic operators with H\"older continuous coefficients}

\author[Mourad Choulli]{Mourad Choulli}
\address{Universit\'e de Lorraine, 34 cours L\'eopold, 54052 Nancy cedex, France}
\email{mourad.choulli@univ-lorraine.fr}

\author{Giorgio Metafune}
\address{Dipartimento di Matematica e Fisica Ennio De Giorgi, Universit\`a del Salento, Via Arnesano, 73100 Lecce, Italy}
\email{giorgio.metafune@unisalento.it}

\thanks{MC is supported by the grant ANR-17-CE40-0029 of the French National Research Agency ANR (project MultiOnde). }

\date{}

\begin{document}

\begin{abstract}
We establish  two-sided Gaussian bounds for fundamental solutions of  general non-divergence form parabolic operators with H\"older continuous coefficients. The result we obtain is essentially based on parametrix method.
\end{abstract}



\maketitle


\section{Introduction}

\subsection{Statement of the main result}

The tremendous  literature on Gaussian bounds for fundamental solutions of second order parabolic operators can be splitted into two classes: divergence or non-divergence operators. In the first class we only quote the deep results obtained by Aronson, following Nash's ideas,  and we refer to  \cite{FS} for a comprehensive treatment. The second class is more classical and can be found in the books \cite{Fr,LSU} where a fundamental solution is constructed, via the parametrix method, assuming H\"older continuity of the coefficients. By  construction the fundamental solution satisfies precise upper bounds but, strangely enough, lower bounds are not proved. In this note we show that the parametrix method produces also lower bounds.

Let $P=\mathbb{R}^n_x\times \mathbb{R}_t$ and set
\[
Q=\{(x,t,\xi ,\tau );\;  (x,t),(\xi ,\tau )\in P,\; \tau <t\}.
\]

The space of continuous and bounded functions $f:P\rightarrow \mathbb{R}$ is denoted by $C_b^0(P)$.

Let $f\in C_b^0(P)$. We say that $f$ is H\"older continuous with exponent $\alpha$, $0<\alpha \le 1$, if
\[
[f]_\alpha =\sup\left\{ \frac{|f(x,t)-f(x',t')|}{|(x-x',t-t')|_\alpha },\; (x,t),\; (x',t') \in P,\; (x,t)\ne (x',t')\right\}<\infty ,
\]
where
\[
|(x-x',t-t')|_\alpha=\left(|x-x'|^2+|t-t'|\right)^{\alpha/2}.
\]

We define
\[
C^\alpha (P)=\{ f\in C_b^0(P);\; [f]_\alpha <\infty\}.
\]
$C^\alpha (P)$ is a Banach space when it is endowed with its natural norm
\[
\|f\|_\alpha =\|f\|_\infty +[f]_\alpha 
\]

and we  also use the notation
\[
\{f\}_\alpha =\sup\left\{\frac{|f(x,t)-f(x',t)|}{|x-x'|^\alpha};\; x,x'\in \mathbb{R}^n,\; x\neq x'\; \mbox{and}\; t\in \mathbb{R}\right\}.
\]

We consider the second order parabolic operator 
\begin{equation} \label{defL}
L=\sum_{i,j=1}^na_{ij}(x,t)\partial ^2_{ij} +\sum_{i=1}^n b_i(x,t) \partial _i +q(x,t)-\partial_t
\end{equation}
with the following assumptions on its coefficients.

\noindent
(a1) $a_{ij}\in C^\alpha (P)$, $1\le i,j\le n$.

\noindent
(a2) The matrix $\mathbf{a}(x,t)=(a_{ij}(x,t))$, $(x,t)\in P$, is symmetric, real-valued,  and there exist constants $\kappa>0$, $M >0$ so that
\[
\kappa |\eta|^2 \le \langle \mathbf{a}(x,t)\eta ,\eta \rangle \le M |\eta |^2,\;\; (x,t)\in P,\; \eta \in \mathbb{R}^n.
\]

\noindent
(a3) $b_i$, $q \in C_b^0(P)$, $1\le i\le n$.

\noindent
(a4) There exists a constant $N_1>0$ so that 
\[ 
\sum_{i,j =1}^n[a_{ij}]_\alpha \le N_1.
\]

\noindent
(a5) There exists a  constant $N_2>0$ so that
\[
\sum_{i=1}^n \|b_i\|_\infty +\|q\|_\infty \le N_2.
\]

\noindent
(a6) $\{b_i\}_\alpha <\infty$, $1\le i\le n$, and $\{q\}_\alpha <\infty$.

Henceforth  we use for convenience the notation $\mathfrak{D}$ for $(n,\alpha ,N_1,N_2, M,\kappa )$.

In this paper the fundamental solution constructed by the parametrix method is denoted by $E=E(x,t;\xi ,\tau )$, $(x,t,\xi ,\tau )\in Q$. Recall that $E$ is a fundamental solution if $E\in C^2(Q)$, $LE=0$ and
\[
\lim_{t\rightarrow \tau}\int_{\mathbb{R}^n} E(x,t;\xi ,\tau )f(\xi )d\xi=f(x),\;\; f\in C_0^\infty (\mathbb{R}^n).
\]

\begin{theorem}\label{theorem-ge}
Let 
\[
c=\frac{1}{8M}\;\;  \mbox{and}\;\; d=\frac{4 \ln \left [ e2^{3n}(M\kappa ^{-1})^{n/2}\Gamma (n/2+1)\right]}{\kappa}.
\]
Under assumptions (a1) to (a6), there exist four constants $\aleph _i=\aleph_i(\mathfrak{D})$, $i=0,1,2,3$, $\aleph_0>0$, $\aleph_1 \ge 0$, $\aleph_2>0$ and $\aleph_3\ge 0$, such that 
\begin{align}
\aleph _0e^{-\aleph _1(t-\tau )}(t-\tau)^{-\frac{n}{2}}e^{-d\frac{|x-\xi|^2}{t-\tau}}\le E(x,&t;\xi ,\tau)\label{ge1}
\\
&\le \aleph _2e^{\aleph _3(t-\tau )}(t-\tau)^{-\frac{n}{2}}e^{-c\frac{|x-\xi|^2}{t-\tau}},\nonumber
\end{align}
for all $(x,t,\xi ,\tau )\in Q$.
\end{theorem}

\begin{remark}\label{rem-ge1}
{\rm
By inspecting the proof of Theorem \ref{theorem-ge} we see that,  in the Gaussian upper bound, we can substitute $c$ by $c^\epsilon=\frac{\epsilon}{4M}$, $0<\epsilon <1$, and $\aleph_i$ by $\aleph _i^\epsilon$, $i=2,3$, with an explicit dependence of $\aleph _2^\epsilon$ and $\aleph _3^\epsilon$ on $\epsilon$.
}
\end{remark}

\subsection{Consequences}

Let $\Omega$ be a $C^{1,1}$-bounded domain of $\mathbb{R}^n$. We denote the parabolic Dirichlet-Green (resp. Neumann-Green) function on $\Omega$ by $G_\Omega ^D$ (resp. $G_\Omega^N$). 

It is well known that, according to the maximum principle, $0\le G_\Omega ^D\le E$. Therefore as a consequence of Theorem \ref{theorem-ge}, we have

\begin{corollary}\label{corollary-ge1}
Let the coefficients of $L$ satisfy assumptions (a1) to (a6). Then the  Dirichlet-Green function $G_\Omega^D$ satisfies
\[
0\le G_\Omega ^D(x,t;\xi ,\tau)\le \aleph _2e^{\aleph _3(t-\tau )}(t-\tau)^{-\frac{n}{2}}e^{-c\frac{|x-\xi|^2}{t-\tau}}, \quad (x,t,\xi ,\tau )\in Q,
\]
where the constants in this inequality are the same as in Theorem \ref{theorem-ge}.
\end{corollary}

We say that $\Omega$ satisfies the chain condition if there exists a constant $\varpi >0$ such that for any two points $x$, $y\in \Omega$ and for any positive integer $m$ there exists a sequence $(x_i)_{0\leq i\leq m}$ of points in $\Omega$ such that $x_0=x$, $x_m=y$ and 
\[
|x_{i+1}-x_i|\leq \frac{\varpi}{m}|x-y|,\;\; i=0,\ldots ,m-1.
\]
The sequence $(x_i)_{0\leq i\leq m}$ is named a chain connecting $x$ and $y$.

Since any bounded Lipschitz domain has the chain condition (this fact can be easily deduced from \cite[Corollary A.1]{CY}), an  adaptation of the proof of \cite[Theorem 3.1]{Ch} (see also \cite{CK}) and the reproducing property enable us to get the following result. 

\begin{corollary}\label{corollary-ge2}
If  the coefficients of $L$ satisfy assumptions (a1) to (a6) then there exist five constants $c_0=c_0(\mathfrak{D})$ and $\aleph _i=\aleph_i(\mathfrak{D})>0$, $i=0,1,2,3$,  such that 
\begin{align*}
\aleph _0e^{-\aleph _1(t-\tau )}(t-\tau)^{-\frac{n}{2}}e^{-c_0\frac{|x-\xi|^2}{t-\tau}}\le G_\Omega ^N(x,&t;\xi ,\tau)
\\
&\le \aleph _2e^{\aleph _3(t-\tau )}(t-\tau)^{-\frac{n}{2}}e^{-c\frac{|x-\xi|^2}{t-\tau}}, 
\end{align*}
for all $(x,t,\xi ,\tau )\in Q$, where $c$ is as  in Theorem \eqref{theorem-ge}.
\end{corollary}

\section{Preliminaries}

In this section the coefficients of $L$ satisfy assumptions (a1) to (a5).

\subsection{Basic properties of generalized Gaussian kernels}

In the sequel we frequently use
\begin{equation}\label{GI}
\int_{\mathbb{R}}e^{-\rho ^2}d\rho =\sqrt{\pi}.
\end{equation}

The Gaussian heat kernel is defined as follows
\begin{equation}\label{hk1}
G(x,t)=\frac{1}{(4\pi t)^{\frac{n}{2}}}e^{-\frac{|x|^2}{4t}},\quad x\in \mathbb{R}^n,\; t>0.
\end{equation}

We have, according to Fubini's theorem,
\[
\int_{\mathbb{R}^n}G(x,t)dx=\left(\int_{\mathbb{R}} \frac{1}{2\sqrt{\pi t}}e^{-\frac{y ^2}{4t}}dy \right)^n,\quad t>0.
\]
Then the change of variable $\rho= \frac{y}{2\sqrt{t}}$ yields
\begin{equation}\label{hk2}
\int_{\mathbb{R}^n}G(x,t)dx=1,\quad t>0,
\end{equation}
where we used the value of the Gauss integral \eqref{GI}.

If $\mathbf{a}=(a^{ij})$ is  $n\times n$ symmetric positive definite matrix, we define the generalized Gaussian heat kernel by
\begin{equation}\label{hk3}
G_{\mathbf{a}}(x,t)= \frac{\sqrt{\mbox{det}\, \mathbf{a}}}{(4\pi t)^{\frac{n}{2}}}e^{-\frac{\langle \mathbf{a}x,x\rangle }{4t}},\;\; x\in \mathbb{R}^n,\; t>0.
\end{equation}

Let $\mathbf{d}=\mbox{diag}(d_1,\ldots ,d_n)$ be a  diagonal matrix and $\mathbf{u}$ an orthogonal matrix, that is $\mathbf{u}^t\mathbf{u}=I$, so that $\mathbf{u}\mathbf{a}\mathbf{u}^t=\mathbf{d}$. Then 
\[
\langle \mathbf{a}x,x\rangle =\langle \mathbf{d}\mathbf{u}x,\mathbf{u}x\rangle, \quad  \mbox{det}\, \mathbf{a}=\prod_{i=1}^n d_i
\]
and 
\[
\int_{\mathbb{R}^n}G_{\mathbf{a}}(x,t)dx=\int_{\mathbb{R}^n}\frac{\sqrt{\mbox{det}\, \mathbf{a}}}{(4\pi t)^{\frac{n}{2}}}e^{-\frac{\langle \mathbf{d}\mathbf{u}x,\mathbf{u}x\rangle}{4t}}dx,\quad t>0.
\]
Since $|\mbox{det}\, \mathbf{u}|=1$, the change of variable $y=\mathbf{u}x$ gives
\[
\int_{\mathbb{R}^n}G_{\mathbf{a}}(x,t)dx=\int_{\mathbb{R}^n}\frac{\sqrt{\mbox{det}\, \mathbf{a}}}{(4\pi t)^{\frac{n}{2}}}e^{-\frac{\langle\mathbf{d}x,x\rangle}{4t}}dx,\quad t>0.
\]
Applying again Fubini's theorem, we get
\begin{align}
\int_{\mathbb{R}^n}G_{\mathbf{a}}(x,t)dx&=\sqrt{\mbox{det}\, \mathbf{a}}\prod_{j=1}^n \int_{\mathbb{R}}\frac{1}{2\sqrt{\pi t}}e^{-\frac{d_i\rho ^2}{4t}}d\rho \label{hk4-}
\\
& =\sqrt{\mbox{det}\, \mathbf{a}}\prod_{j=1}^n \int_{\mathbb{R}}\frac{1}{2\sqrt{d_i\pi t}}e^{-\frac{\rho ^2}{4t}}d\rho \nonumber
\\
& =\prod_{j=1}^n \int_{\mathbb{R}}\frac{1}{2\sqrt{\pi t}}e^{-\frac{\rho ^2}{4t}}d\rho=1,\quad t>0.\nonumber
\end{align}

It is straightforward to check that $G_{\mathbf{a}}\in C^\infty (\mathbb{R}^n \times (0,\infty ))$ and, since
\[
\partial _k \langle \mathbf{a}x,x\rangle =2\sum_{j=1}^na^{kj}x^j =2(\mathbf{a}x)_k,\quad x\in \mathbb{R}^n,
\]
we have
\begin{equation}\label{hk4}
\partial _k G_{\mathbf{a}}(x,t)= -\frac{1}{2t}G_{\mathbf{a}}(x,t )(\mathbf{a}x)_k,\quad x\in \mathbb{R}^n,\; t>0.
\end{equation}
We easily derive from \eqref{hk4}
\begin{equation}\label{hk5}
\partial ^2_{k\ell} G_{\mathbf{a}}(x,t)= \frac{1}{4t^2}G_{\mathbf{a}}(x,t )(\mathbf{a}x)_k(\mathbf{a}x)_\ell -\frac{1}{2t}G_{\mathbf{a}}(x,t )a^{k\ell},\quad x\in \mathbb{R}^n,\; t>0.
\end{equation}

Let $\mathbf{a}^{-1}=(a_{ij})$. Inserting the identity 
\[
\sum_{k,\ell =1}^na_{k\ell}(\mathbf{a}x)_k(\mathbf{a}x)_\ell=  \langle\mathbf{a}^{-1}\mathbf{a}x, x\rangle=\langle  \mathbf{a}x,x\rangle 
\]
in \eqref{hk5} we obtain
\begin{equation}\label{hk6}
\sum_{k,\ell =1}^na_{k\ell}\partial ^2_{k\ell} G_{\mathbf{a}}(x,t)= \left(\frac{1}{4t^2}\langle \mathbf{a}x,x\rangle -\frac{n}{2t}\right)G_{\mathbf{a}}(x,t ),\quad x\in \mathbb{R}^n,\; t>0.
\end{equation}
On the other hand, it is straightforward to check that
\begin{equation}\label{hk7}
\partial _tG_{\mathbf{a}}(x,t) = \left(\frac{1}{4t^2}\langle \mathbf{a}x,x\rangle -\frac{n}{2t}\right)G_{\mathbf{a}}(x,t ),\quad x\in \mathbb{R}^n,\; t>0.
\end{equation}

We define the parabolic operator $L_{\mathbf{a}^{-1}}$ by
\[
L_{\mathbf{a}^{-1}}= \sum_{i,j =1}^na_{ij}\partial ^2_{ij} -\partial _t.
\]

Comparing \eqref{hk6} and \eqref{hk7} we see that $G_{\mathbf{a}}$ satisfies
\begin{equation}\label{hk8}
L_{\mathbf{a}^{-1}}G_{\mathbf{a}}(x,t)=0 ,\quad x\in \mathbb{R}^n,\; t>0.
\end{equation}

\subsection{The parametrix}

Let $\mathbf{a}^{-1}(x,t)=(a^{ij}(x,t))$, $(x,t)\in P$, where $(a^{ij}(x,t))$ is the inverse of the matrix $(a_{ij}(x,t))$, and define
\[
Z(x,t;\xi ,\tau )=G_{\mathbf{a}^{-1}(\xi ,\tau)}(x-\xi ,t-\tau ),\quad (x,t,\xi,\tau )\in Q,
\]
that is
\begin{equation}\label{fs1}
Z(x,t;\xi ,\tau )= \frac{\sqrt{\mbox{det}\,\mathbf{a}^{-1}(\xi ,\tau)}}{(4\pi (t-\tau))^{\frac{n}{2}}}e^{-\frac{\langle \mathbf{a}^{-1}(\xi ,\tau )(x-\xi ),(x-\xi )\rangle }{4(t-\tau )}},\quad (x,t,\xi,\tau )\in Q.
\end{equation}

This function is usually called the parametrix associated to the parabolic operator $L$. According to the results of the previous subsection, for any $(\xi ,\tau )\in P$, $Z(\cdot ,\cdot ;\xi ,\tau )\in C^\infty (P_\tau )$ with $P_\tau =\{(x,t)\in \mathbb{R}^n;\; t>\tau\}$, and
\begin{equation}\label{fs2}
\sum_{i,j=1}^na_{ij}(\xi ,\tau)\partial _{ij}^2Z(\cdot ,\cdot ;\xi ,\tau ) -\partial _tZ(\cdot ,\cdot ;\xi ,\tau )=0\;\; \mbox{in}\; P_\tau .
\end{equation}

Let us define
\begin{align*}
&d_i(x,t;\xi ,\tau )=-\frac{1}{2(t-\tau)}\sum_{j=1}^na^{ij}(\xi ,\tau )(x_j-\xi _j),
\\
&
d_{ij}(x,t;\xi ,\tau)=-\frac{a^{ij}(\xi ,\tau )}{2(t-\tau)}+d_i(x,t;\xi ,\tau)d_j(x,t;\xi ,\tau).
\end{align*}

From \eqref{hk4}  and \eqref{hk5} we have
\[
\partial _iZ=d_iZ\;\; \mbox{and}\;\; \partial_{ij}^2Z=d_{ij}Z.
\]
Therefore, taking into account \eqref{fs2}, we have
\begin{equation}\label{fs3}
LZ=\left[ \sum_{i,j=1}^n\left(a_{ij}(x,t)-a_{ij}(\xi ,\tau )\right)d_{ij}+\sum_{i=1}^nd_ib_i+q\right]Z=\Psi Z,
\end{equation}
where 
\[
\Psi = \sum_{i,j=1}^n\left(a_{ij}(x,t)-a_{ij}(\xi ,\tau )\right)d_{ij} +\sum_{i=1}^nd_ib_i+q.
\]

We need a pointwise estimate for $LZ$. To this end, we start with the following lemma

\begin{lemma}\label{pa1}
We have
\begin{equation}\label{fs4}
|\mathbf{a}^{-1}(x,t)\eta |\le \frac{1}{\kappa}|\eta |,\quad (x,t)\in P,\; \eta \in \mathbb{R}^n ,
\end{equation}

\begin{equation}\label{fs8+}
\sup_{1\le i,j\le n}\|a^{ij}\|_\infty \le \frac{1}{\kappa}.
\end{equation}
and 
\begin{equation}\label{fs8++}
\frac{\langle \mathbf{a}^{-1}(x,\tau )(x-\xi ), x-\xi\rangle }{4(t-\tau )}\ge \frac{1}{4M}\frac{|x-\xi]^2}{t-\tau}.
\end{equation}
\end{lemma}

\begin{proof}
From assumption (a2), we have
\[
\langle \mathbf{a}(x,t)\eta ,\eta \rangle \ge \kappa |\eta |^2,\;\; (x,t)\in P,\; \eta \in \mathbb{R}^n.
\]
In this inequality we get by substituting $\eta$ by $\mathbf{a}^{-1}(x,t)\eta$
\[
|\mathbf{a}^{-1}(x,t)\eta ||\eta |\ge \langle \mathbf{a}^{-1}(x,t)\eta ,\eta \rangle \ge \kappa |\mathbf{a}^{-1}(x,t)\eta |^2,\quad (x,t)\in P,\; \eta \in \mathbb{R}^n
\]
and \eqref{fs4} follows. 

Since $a^{ij}=\langle \mathbf{a}^{-1} \mathbf{e}_i, \mathbf{e}_j\rangle $, where $(\mathbf{e}_1,\ldots ,\mathbf{e}_n)$ the canonical basis of $\mathbb{R}^n$, \eqref{fs8+} follows from \eqref{fs4}.

Finally, \eqref{fs8++} is equivalent to $\langle \mathbf{a}^{-1}(x,\tau )\eta, \eta \rangle \ge \frac{1}{M}|\eta|^2$  or $\langle \mathbf{a}(x,\tau )\eta, \eta \rangle \le M|\eta|^2$, which holds by assumption.
\end{proof}

From \eqref{fs4}, we get
\begin{equation}\label{fs9}
\|d_i\|_\infty \le \frac{|x-\xi |}{2\kappa (t-\tau)}\quad  {\rm or } \quad \|d_i\|_\infty \le \frac{\varrho}{2\kappa \sqrt{t-\tau}}
\end{equation}
where
\[
\varrho =\frac{|x-\xi |}{\sqrt{t-\tau}}.
\]

It is easy to see that \eqref{fs8+}  and \eqref{fs9} entail
\begin{equation}\label{fs10}
\|d_{ij}\|_\infty \le \left(\frac{1}{2\kappa }+\frac{\varrho ^2}{4\kappa^2}\right)\frac{1}{t-\tau }.
\end{equation}

Hence
\begin{equation}\label{fs11}
\left|\sum_{i,j=1}^n\left(a_{ij}(x,t)-a_{ij}(\xi ,\tau )\right)d_{ij}\right| \le N_1\left(\frac{1}{2\kappa }+\frac{\varrho ^2}{4\kappa^2}\right)\frac{(1+\varrho ^2 )^{\frac{\alpha}{2}}}{(t-\tau)^{1-\frac{\alpha}{2}}}.
\end{equation}

On the other hand, we get from \eqref{fs9} 
\begin{equation}\label{fs9.1}
\left| \sum_{i=1}^n b_id_i +q  \right|\le  N_2 \left (\frac{\varrho}{2k \sqrt{t-\tau}}+1\right ) \le N_2\frac{1+\frac{\varrho}{2\kappa}}{(t-\tau)^{1-\frac{\alpha}{2}}},\quad t-\tau \le 1.
\end{equation}

In light of \eqref{fs11} and \eqref{fs9.1} we obtain
\begin{equation}\label{fs12}
\|\Psi \|_\infty \le N_1\left(\frac{1}{2\kappa }+\frac{\varrho ^2}{4\kappa^2}\right)\frac{(1+\varrho ^2 )^{\frac{\alpha}{2}}}{(t-\tau)^{1-\frac{\alpha}{2}}}
+N_2\frac{1+\frac{\varrho}{2\kappa}}{(t-\tau)^{1-\frac{\alpha}{2}}} ,\quad t-\tau \le 1.
\end{equation}

Now \eqref{fs8++} implies
\begin{equation}\label{fs13}
|Z(x,t)|\le \frac{1}{ (4\kappa \pi (t-\tau ))^{\frac{n}{2}}}e^{-\frac{1}{4M}\varrho ^2}.
\end{equation}

Recall that $c=\frac{1}{8M}$ and let 
\begin{equation}\label{e}
C=\frac{1}{ (4\kappa\pi )^{\frac{n}{2}}}\max_{\lambda >0}\left[N_1\left(\frac{1}{2\kappa }+\frac{\lambda ^2}{4\kappa^2}\right)(1+\lambda ^2 )^{\alpha/2}
+N_2\left(\frac{\lambda}{\kappa}+1\right)\right]e^{-c\lambda ^2}.
\end{equation}

 If $\Phi_1=LZ=\Psi Z$, then a combination of \eqref{fs12} and \eqref{fs13} gives
\begin{equation}\label{fs14}
|LZ|=|\Psi Z|\le C(t-\tau)^{-\frac{n}{2}-1+\beta }e^{-c\varrho ^2},\;\; t-\tau \le 1,
\end{equation}
with   $\beta =\frac{\alpha}{2}$.

\section{Two-sided Gaussian bounds}

In this section the coefficients of $L$ satisfy (a1) to (a6). Let $\Phi_1=LZ$,
\[
 \Phi_{\ell +1}(x,t,\xi ,\tau)= \int_\tau ^t\int_{\mathbb{R}^n}\Phi_1(x,t;\eta ,\sigma )\Phi_\ell (\eta ,\sigma ,\xi ,\tau )d\eta d\sigma ,\quad \ell \ge 1
 \]
and define
\[
\Phi=\sum_{\ell \ge 1}\Phi_\ell .
\]

Let $E$ be the fundamental solution, associated to $L$, constructed by the parametrix method. According to \cite{Fr,LSU}, $E$ is given by
\begin{equation}\label{fs18}
E(x,t;\xi,\tau )=Z(x,t;\xi ,\tau )+\int_\tau^t\int_{\mathbb{R}^n}Z(x,t;\eta ,\sigma )\Phi (\eta ,\sigma ; \xi ,\eta )d\eta d\sigma ,
\end{equation}
for all $(x,t,\xi ,\tau )\in Q$.

We refer to \cite[Chapter 1]{Fr} or to \cite[Chapter IV]{LSU} for more details.

\subsection{Preliminary estimate}

The following lemma will be useful in the sequel.

 \begin{lemma}\label{lemma2.1}$($\cite[Chapter 1, Section 4]{Fr}$)$
 Let $\lambda >0$ and $-\infty <\gamma,\delta  <1$. Then
 \begin{align*}
 \int_\tau ^t\int_{\mathbb{R}^n}(t-\sigma )^{-\frac{n}{2}-\gamma}e^{-\frac{\lambda |x-\eta |^2}{t-\sigma }}&(\sigma -\tau )^{-\frac{n}{2}-\delta}e^{-\frac{\lambda |\eta -\xi |^2}{\sigma -\tau}}d\eta d\sigma 
\\
&=\left(\frac{4\pi}{\lambda }\right)^{\frac{n}{2}}B\left(1-\gamma,1-\delta \right)(t-\tau)^{-\frac{n}{2}+1-\gamma-\delta}e^{-\frac{\lambda |x-\xi |^2}{t-\tau }},
 \end{align*}
 where $B$ is the Euler beta function.
 \end{lemma}
 
Let $C$ be the constant given by \eqref{e} and assume that $t-\tau \le 1$. We deduce from \eqref{fs14} 
 \begin{equation}\label{fs14+}
|\Phi_1|\le C(t-\tau)^{-\frac{n}{2}-1+\beta }e^{-c\varrho ^2}.
\end{equation}
 
 Let $\widetilde{C}= \left(\frac{4\pi}{c }\right)^{\frac{n}{2}}$. We have by applying Lemma \ref{lemma2.1}
 \[
 |\Phi_2|\le \widetilde{C} C^2B(\beta ,\beta )(t-\tau)^{-\frac{n}{2}-1+2\beta}e^{-c\varrho ^2}.
 \]
 By induction in $\ell$, wo obtain
 \[
 |\Phi_\ell |\le \widetilde{C}^{\ell -1} C^{\ell}\prod_{j=1}^{\ell -1}B(\beta ,j\beta )(t-\tau)^{-\frac{n}{2}-1+\ell \beta}e^{-c\varrho ^2},\quad \ell \ge 2.
 \]
 
If $\Gamma$ is the Euler gamma function, we recall that
 \[
B(\beta ,j\beta )=\frac{\Gamma (\beta )\Gamma (j\beta )}{\Gamma ((j+1)\beta )}.
 \]
 Therefore
\[
\prod_{j=1}^{\ell -1}B(\beta ,j\beta )=\frac{\Gamma (\beta)^\ell }{\Gamma (\ell \beta )}
\]
and hence
\[
|\Phi_\ell |\le \widetilde{C}^{-1} \frac{\Lambda ^\ell}{\Gamma (\ell \beta )}(t-\tau)^{-\frac{n}{2}-1+\ell \beta}e^{-c\varrho ^2},\quad \ell \ge 2,
\]
where  $\Lambda =C\widetilde{C}\Gamma (\beta)$. Since $t-\tau \le 1$, we obtain
\begin{equation}\label{fs15}
|\Phi_\ell |\le \widetilde{C}^{-1} \frac{\Lambda ^\ell}{\Gamma (\ell \beta )}(t-\tau)^{-\frac{n}{2}-1+ \beta}e^{-c\varrho ^2},\quad \ell \ge 2,
\end{equation}

If $\overline{C}=\widetilde{C}^{-1} $, then \eqref{fs15} takes the form
\begin{equation}\label{fs16}
|\Phi_\ell |\le\overline{C}\frac{\Lambda ^\ell}{\Gamma (\ell \beta )}(t-\tau)^{-\frac{n}{2}-1+ \beta}e^{-c\varrho ^2},\quad \ell \ge 2.
\end{equation}

From Stirling's formula for the $\Gamma$ function (see for instance \cite[Chapter V, Section 3]{LV}) we have
\[
\Gamma (x+1 )\sim x^x e^{-x}\sqrt{2\pi x}, \quad x \to \infty .
\]

Therefore, the series
\begin{equation} \label{S}
S=C+\overline{C}\sum_{\ell \ge 2}\frac{\Lambda ^\ell}{\Gamma (\ell \beta )}
\end{equation}
is convergent.

We get from \eqref{fs14} and \eqref{fs16}
\begin{equation}\label{fs17}
|\Phi |\le S(t-\tau)^{-\frac{n}{2}-1+ \beta}e^{-c\varrho ^2}.
\end{equation}

\subsection{The upper bound}

In light of \eqref{fs14} and \eqref{fs17}, Lemma \ref{lemma2.1} yields
\begin{equation}\label{fs19}
\left|  \int_\tau^t\int_{\mathbb{R}^n}Z(x,t;\eta ,\sigma )\Phi (\eta ,\sigma ; \xi ,\tau )d\eta d\sigma\right|\le \frac{SB(1,\beta )}{(\kappa c)^{\frac{n}{2}}}(t-\tau)^{-\frac{n}{2}+ \beta}e^{-c\varrho ^2},
\end{equation}
for all $(x,t,\xi ,\tau )\in Q$ and  $t-\tau \le 1$.

Let 
\[
\widehat{C}=\frac{1}{(4\kappa \pi )^{\frac{n}{2}}}+\frac{SB(1,\beta )}{(\kappa c)^{\frac{n}{2}}}.
\]

As an immediate consequence of \eqref{fs14} and \eqref{fs19}, we have
\begin{equation}\label{fs20}
|E(x,t;\xi ,\tau )|\le \widehat{C}(t-\tau)^{-\frac{n}{2}}e^{-c\varrho ^2},\quad (x,t,\xi ,\tau )\in Q,\; t-\tau \le 1.
\end{equation}

We recall that $E$ possesses the so-called reproducing property
\begin{equation}\label{rp}
E(x,t;\xi ,\tau )=\int_{\mathbb{R}^n}E(x,t;\eta ,\sigma )E(\eta ,\sigma ;\xi ,\tau )\, d\eta,\quad \tau< \sigma < t .
\end{equation}

Applying \eqref{fs20}, we get
\begin{equation}\label{fs21}
|E(x,t,\xi ,\tau )|\le \widehat{C}^2\int_{\mathbb{R}^n}(t-\sigma )^{-\frac{n}{2}}e^{-c\frac{|x-\eta |^2}{4(t-\sigma )}}(\sigma -\tau)^{-\frac{n}{2}}e^{-c\frac{|\eta -\xi|^2}{4(\sigma -\tau )}}d\eta ,
\end{equation}
for all $t-\tau \le 2$, where $\sigma =\frac{t+\tau}{2}$.

We introduce a variable $z$ so that
\[
c\frac{|x-\eta |^2}{4(t-\sigma )}+c\frac{|\eta -\xi|^2}{4(\sigma -\tau )}=c\frac{|x -\xi|^2}{4(t -\tau )}+|z|^2.
\]
Using the identity $|x-\eta |^2=|x-\xi |^2+|\xi -\eta |^2+\langle x-\xi ,\xi -\eta \rangle$, we get
\begin{align*}
\frac{|x-\eta |^2}{t-\sigma }+\frac{|\eta -\xi|^2}{\sigma -\tau }&-\frac{|x -\xi|^2}{t -\tau }
\\
&=\frac{(\sigma -\tau )|x-\xi |^2}{(t-\sigma)(t-\tau)} +\frac{(t-\tau )|\eta -\xi| ^2}{(t-\sigma )(\sigma -\tau )} +\frac{2\langle x-\xi ,\xi -\eta \rangle}{(t-\sigma)^2} .
\\
&=\left| \left(\frac{\sigma -\tau}{(t-\sigma)(t-\tau)}\right)^{\frac{1}{2}} (x-\xi ) +\left(\frac{t-\tau}{(t-\sigma )(\sigma -\tau )}\right)^{\frac{1}{2}}(\xi -\eta ) \right|^2 .
\end{align*}
Therefore, we can for instance take
\[
z=\left(c\frac{t-\tau}{t-\sigma}\right)^{\frac{1}{2}}\frac{\eta -\xi}{2(\sigma -\tau)^{\frac{1}{2}}}+\left(c\frac{\sigma -\tau}{t-\sigma}\right)^{\frac{1}{2}}\frac{\xi -x}{2(t -\tau)^{\frac{1}{2}}}.
\]

Passing to the variable $z$ in \eqref{fs21}, we deduce
\[
|E(x,t,\xi ,\tau )|\le \widetilde{C} \widehat{C}^2 (t-\tau )^{-\frac{n}{2}}e^{-c\varrho ^2},\quad t-\tau \le 2.
\]

Next assume that $t-\tau >2$ and let $m$ be the smallest integer so that $t-\tau \le m$. Define
\[
\sigma _0=\tau ,\quad \sigma_1=\tau +\frac{t-\tau}{m},\ldots ,\sigma_{m-1}=\tau+(m-1)\frac{t-\tau}{m},\quad \sigma_m=t.
\]

Iterating the reproducing property \eqref{rp}, we get
\begin{align*}
E(x,t;\xi ,\tau )=\int_{\mathbb{R}^n}\ldots\int_{\mathbb{R}^m} E(x,\sigma_m,\eta _m,\sigma_{m-1})E(\eta_m ,&\sigma_{m-1} ,\eta_{m-1} ,\sigma_{m-2})
\\
&\ldots E(\eta _1,\sigma_1,\xi ,\sigma_0)d\eta _1\ldots d\eta_m.
\end{align*}
Repeating inductively the case $m=2$, we find
\[
|E(x,t,\xi ,\tau )|\le \widetilde{C}^{m-1} \widehat{C}^{m} (t-\tau )^{-\frac{n}{2}}e^{-c\varrho ^2}.
\]
This and the fact that $m<t-\tau+1$ entail
\[
|E(x,t,\xi ,\tau )|\le \widetilde{C}^{-1} e^{\max\left(0,\ln (\widetilde{C}\widehat{C})\right)}e^{\max\left(0,\ln (\widetilde{C}\widehat{C})\right)(t-\tau )} (t-\tau )^{-\frac{n}{2}}e^{-c\varrho ^2}.
\]
This is the expected Gaussian upper bound.

A more precise upper bound can be obtained by optimizing the constants appearing in the previous computations. We do it in the special case $b_i=q=0$, where the iteration procedure based on \eqref{rp} is not needed.

\begin{corollary} \label{preciseupper}
If $b_i=q=0$, then 
\[
E(x,t;\xi, \tau) \le \frac{1}{(4\kappa \pi)^{\frac{n}{2}}}(t-\tau)^{-\frac{n}{2}}e^{-\frac{\varrho^2}{4M}}\left (1+c_1(t-\tau)^{\frac{\alpha}{2}}e^{c_2 \left ((t-\tau)+ \varrho^\gamma \right)}\right ),
\]
for all $(x,t,\xi ,\tau )\in Q$, where $\varrho =\frac{|x-\xi |}{\sqrt{t-\tau}}$ and $\gamma=\frac{4\alpha+8}{3\alpha+4}<2$.
\end{corollary}

\begin{proof}
First we note that the restriction $t-\tau \le 1$ is not needed in \eqref{fs12}, since it comes from \eqref{fs9.1} only. Then we define $C_\epsilon$ as in \eqref{e} with $c=\frac{\epsilon}{4M}$, $N_2=0$. It is easy to see that  $C_\epsilon \le A\epsilon^{-2-\alpha}$ with $A>0$ and this leads to  \eqref{fs14} with this $C_\epsilon$ and $c=\frac{(1-\epsilon)}{4M}$. Next we write \eqref{fs16} with $\ell \beta$ instead of $\beta$, since we no longer assume that $t-\tau \le 1$.

Entering this estimate in the constants $C, \Lambda$  defining   $S$ (see \eqref{S}), using \cite[Theorem 2, Section 15, Chapter V]{Chabat} and Stirling's formula again, we deduce that
\[
\sum_{\ell \ge 2} \frac{\Lambda^\ell (t-\tau)^{\ell \beta} }{\Gamma(\ell \beta)}\le c_1  (t-\tau)^{2\beta}e^{c_2((t-\tau)+\Lambda^{\frac{1}{\beta}})}
\]
and $S \le c_1 e^{c_2 ((t-\tau)+\epsilon^{-(2+\frac{4}{\alpha})})}$.  Then we use this estimate in  \eqref{fs19} with $c=\frac{(1-\epsilon)}{4M}$ to get 
\begin{align*}
&\left|  \int_\tau^t\int_{\mathbb{R}^n}Z(x,t;\eta ,\sigma )\Phi (\eta ,\sigma ; \xi ,\eta )d\eta d\sigma\right|
\\
&\hskip 3cm \le c_1(t-\tau)^{-\frac{n}{2}+\beta}e^{-\frac{(1-\epsilon)}{4M}\varrho ^2+c_2\epsilon^{-(2+\frac{4}{\alpha})}+c_2(t-\tau)}.
\end{align*}
Optimizing over $\epsilon$ and using \eqref{fs18}, the corollary follows. 
\end{proof}

\subsection{The lower bound}

From the previous analysis, we easily get
\[
Z(x,t;\xi ,\tau )\ge \frac{1}{(4\pi M )^{\frac{n}{2}}}(t-\tau)^{-\frac{n}{2}}e^{-\frac{1}{\kappa}\rho ^2}.
\]
Hence,
\begin{equation}\label{fs22}
Z(x,t;\xi ,\tau )\ge \frac{e^{-1}}{(4\pi M )^{\frac{n}{2}}}(t-\tau)^{-\frac{n}{2}},\quad  |x-\xi |^2\le \kappa (t-\tau ).
\end{equation}

A combination of \eqref{fs19} and \eqref{fs22} yields
\[
E(x,t;\xi ,\tau )\ge \frac{e^{-1}}{(4\pi M )^{\frac{n}{2}}}(t-\tau)^{-\frac{n}{2}}-\frac{SB(1,\beta )}{(\kappa c)^{\frac{n}{2}}}(t-\tau)^{-\frac{n}{2}+ \beta} ,
\]
for all $|x-\xi |^2\le \kappa (t-\tau )$ and   $t-\tau \le 1$.

Fix $\delta \le 1$ sufficiently small in such a way that
\[
\frac{e^{-1}}{(4\pi M )^{\frac{n}{2}}}-\frac{SB(1,\beta )}{(\kappa c)^{\frac{n}{2}}}\delta ^\beta\ge \frac{e^{-1}}{2(4\pi M )^{\frac{n}{2}}}.
\]
Then, with $\mu=\frac{e^{-1}}{2(4\pi M )^{\frac{n}{2}}}$,
\begin{equation}\label{fs23}
E(x,t;\xi ,\tau )\ge \mu (t-\tau)^{-\frac{n}{2}} ,\quad  |x-\xi |^2\le \kappa (t-\tau ),\; t-\tau\le \delta .
\end{equation}

Let $x$ and $\xi$ be given so that $2|x-\xi | >\sqrt{\kappa (t-\tau)}$ and let $m\ge 2$ be the smallest integer so that
\begin{equation}\label{fs24}
\frac{4|x-\xi |^2}{m}\le \kappa (t-\tau).
\end{equation}

Define the sequence $(x_k)_{0\le k\le m}$
\[
x_k=x+\frac{k}{m}(\xi -x ),\quad 0\le k\le m.
\]
Set
\[
r=\frac{1}{4}\frac{\sqrt{\kappa (t-\tau)}}{\sqrt{m}}
\]
and
\[
\sigma _k=\tau +\frac{k}{m}(t-\tau ),\quad 0\le k\le m.
\]
Using \eqref{fs23}, the positivity of $E$ and the reproducing property, we get
\begin{align*}
E(x,t;&\xi ,\tau)
\\
&\ge \mu ^m\int_{B(x_1,r)}\ldots \int_{B(x_{m-1},r)} (\sigma _1-\sigma_0)^{-\frac{n}{2}}\ldots (\sigma _m-\sigma_{m-1})^{-\frac{n}{2}}d\eta_1\ldots d\eta_{m-1},
\end{align*}
where we used
\[
|x_{i+1}-x_i|=\frac{1}{\sqrt{m}}\frac{|x-\xi|}{\sqrt{m}}\le \frac{1}{2}\frac{\sqrt{\kappa (t-\tau)}}{\sqrt{m}}=2r,
\]
and
\begin{align*}
|\eta _{i+1}-\eta_i|\le |\eta_{i+1}&-x_{i+1}|+|x_{i+1}-x_i|+|x_i-\eta _i|
\\
&<2r+|x_{i+1}-x_i|\le 4r=\frac{\sqrt{\kappa (t-\tau)}}{\sqrt{m}}=\sqrt{\kappa (\sigma_{i+1}-\sigma_i)}.
\end{align*}
Whence
\[
E(x,t;\xi ,\tau )\ge \kappa ^{-\frac{n}{2}}\nu ^m (t-\tau)^{-\frac{n}{2}},
\]
with
\[
\nu =\frac{\kappa^{\frac{n}{2}}}{e M^{\frac{n}{2}}2^{3n}\Gamma (n/2+1)} <1.
\]
Noting that
\[
m<\frac{4|x-\xi|^2}{\kappa (t-\tau)}+1,
\]
we obtain
\begin{align*}
E(x,t;\xi ,\tau )\ge \kappa ^{-\frac{n}{2}}&e^{-|\ln \nu |m}(t-\tau)^{-\frac{n}{2}}
\\
&\ge \kappa ^{-\frac{n}{2}}e^{-|\ln \nu |}(t-\tau)^{-\frac{n}{2}}e^{-\frac{4|\ln \nu |}{\kappa} \frac{|x-\xi|^2}{t-\tau}},\quad t-\tau \le \delta .
\end{align*}

If $C_0=\min \left(\mu ,\kappa ^{-\frac{n}{2}}e^{-|\ln \nu |}\right)$ and $d=\frac{4|\ln \nu |}{\kappa}$, then the last inequality and \eqref{fs23} yield
\[
E(x,t;\xi ,\tau )\ge C_0(t-\tau)^{-\frac{n}{2}}e^{-d\frac{|x-\xi|^2}{t-\tau}},\quad t-\tau \le \delta .
\]

We now proceed similarly to the case of the upper bound to remove the condition $t-\tau\le \delta$. If $m$ is the smallest integer so that $t-\tau \le m\delta$, we get
\[
E(x,t;\xi ,\tau )\ge \widetilde{C}^{-1}\left(\widetilde{C}C_0\right)^m(t-\tau)^{-\frac{n}{2}}e^{-d\varrho ^2},
\]
from which we deduce
\[
E(x,t;\xi ,\tau )\ge\widetilde{C}^{-1} e^{\min\left(0, \ln (\widetilde{C}C_0)\right)}e^{\min\left(0,\frac{\ln (\widetilde{C}C_0)}{\delta}\right)(t-\tau )} (t-\tau )^{-\frac{n}{2}}e^{-d\varrho ^2}.
\]

\end{document}